\documentclass[12pt,leqno,a4paper]{amsart}

\usepackage{amssymb,amsthm,enumerate}
\usepackage{hyperref}                         
\usepackage[all]{xy}

\hypersetup{pdftex,                           
bookmarks=true,
pdffitwindow=true,
colorlinks=true,
citecolor=black,
filecolor=black,
linkcolor=black,
urlcolor=blue,
hypertexnames=true}

\newcommand{\IN}{{\mathbb{N}}}

\newcommand{\fa}{{\mathfrak{a}}}   
\newcommand{\fp}{{\mathfrak{p}}}     
\newcommand{\fm}{{\mathfrak{m}}}

\newcommand{\cO}{{\mathcal{O}}}

\DeclareMathOperator{\End}{End}               
\DeclareMathOperator{\Hom}{Hom}             
\DeclareMathOperator{\Syl}{Syl}                  
\DeclareMathOperator{\Res}{Res}               
\DeclareMathOperator{\Ind}{Ind}                  

\newcommand{\lconj}[2]{\,^{#1}\!#2}                     

\let\lra=\longrightarrow

\let\wh=\widehat

\newtheorem{thm}{Theorem}[section]
\newtheorem{lem}[thm]{Lemma}

\newtheorem{prop}[thm]{Proposition}

\theoremstyle{definition}

\theoremstyle{remark}
\newtheorem{rem}[thm]{Remark}

\begin{document}


\title[Lifting endo-$p$-permutation modules]
      {Lifting endo-$p$-permutation modules}

\date{\today}

\author{Caroline Lassueur and Jacques Th\'{e}venaz}
\address{Caroline Lassueur\\ FB Mathematik, TU Kaiserslautern, Postfach 3049, 67653 Kaisers\-lautern, Germany.}
\email{lassueur@mathematik.uni-kl.de}
\address{Jacques Th\'{e}venaz\\ EPFL, Section de Math\'{e}matiques, Station 8, CH-1015 Lausanne, Switzerland.}
\email{jacques.thevenaz@epfl.ch}



\begin{abstract}
We prove that all endo-$p$-permutation modules for a finite group are liftable from characteristic $p>0$ to characteristic~$0$.
\end{abstract}

\maketitle


\pagestyle{myheadings}
\markboth{C. Lassueur and J. Th\'{e}venaz}{Lifting endo-$p$-permutation modules}


\section{Introduction}

Throughout we let $p$ be a prime number and  $G$ be a finite group of order divisible by~$p$. We let $\cO$ denote a complete discrete valuation ring of characteristic~$0$ with a residue field $k:=\cO/\fp$ of positive characteristic $p$, were $\fp=J(\cO)$ is the unique maximal ideal of $\cO$. Moreover, for $R\in\{\cO,k\}$ we consider only finitely generated $RG$-lattices.\par

Amongst  finitely generated $kG$-modules very few classes of modules are known to be liftable to $\cO G$-lattices. Projective $kG$-modules are known to lift uniquely, and more generally,  so do $p$-permutation $kG$-modules (see e.g. \cite[\S 2.6]{BENSON84}).
In the special case where the group $G$ is a $p$-group, Alperin \cite{ALPERINlift} proved  that endo-trivial $kG$-modules are liftable, and Bouc \cite[Corollary 8.5]{BOUC06} observed that so are endo-permutation $kG$-modules as a consequence of their classification. \par

Passing to arbitrary groups, it is proved in \cite{LMS16} that Alperin's result extends to endo-trivial modules over arbitrary groups. 
It is therefore legitimate to ask whether Bouc's result may be extended to arbitrary groups. A natural candidate for such a generalisation is the class of so-called \emph{endo-$p$-permutation} $kG$-modules introduced  by Urfer \cite{URFER07}, which are $kG$-modules whose $k$-endomorphism algebra is a $p$-permutation $kG$-module.
We extend this definition to $\cO G$-lattices and prove that any indecomposable endo-$p$-permutation $kG$-module lifts to an endo-$p$-permutation $\cO G$-lattice with the same vertices.\par

We emphasise that our proof relies on a nontrivial result, namely the lifting of endo-permutation modules, which is a consequence of their classification. Moreover, there are two crucial points to our argument:  the first one is the fact that reduction modulo~$\fp$ applied to the class of endo-$p$-permutation $\cO G$-lattices preserves both indecomposability and vertices, while the second one relies on properties of the $G$-algebra structure of the endomorphism ring of endo-permutation $RG$-lattices.

\section{Endo-$p$-permutation lattices}
Recall that an $\cO G$-lattice is an $\cO G$-module which is free as an $\cO$-module.
For $R\in\{\cO,k\}$ an $RG$-lattice $L$ is called a \emph{$p$-permutation lattice} if $\Res^G_P(L)$ is a permutation $RP$-lattice for every $p$-subgroup $P$ of $G$, or equivalently, if $L$ is isomorphic to a direct summand of  a permutation $RG$-lattice.

Following Urfer \cite{URFER07}, we call an $RG$-lattice $L$  an \emph{endo-$p$-permutation}
$RG$-lattice if its endomorphism algebra $\End_R(L)$ is a $p$-permutation $RG$-lattice, where $\End_R(L)$ is endowed with its natural $RG$-module structure via the action of $G$ by conjugation: 
$$\lconj{g}{\phi}(m)=g\cdot\phi(g^{-1}\cdot m)\quad \forall\,g\in G, \forall\, \phi\in \End_R(L)\text{ and }\forall\,m\in L \,.$$
Equivalently, $L$ is an endo-$p$-permutation $RG$-lattice if and only if
$\Res^G_P(L)$  is an endo-permutation $RP$-lattice for a Sylow $p$-subgroup $P\in\Syl_p(G)$, or also if $\Res^G_Q(L)$ is an endo-permutation $RQ$-lattice for every $p$-subgroup $Q$ of~$G$.

This generalises the notion of an \emph{endo-permutation} $R P$-lattice over a $p$-group~$P$, introduced by Dade in \cite{Dade78a,Dade78b}. In fact an $RP$-lattice is an endo-$p$-permutation $RP$-lattice if and only if it is an endo-permutation lattice. 
An endo-permutation $RP$-lattice $M$ is said to be \emph{capped} if it has at least one indecomposable direct summand with vertex~$P$, and in this case there is in fact a unique isomorphism class of  indecomposable direct summands of $M$ with vertex~$P$. Moreover, considering an equivalence relation called \emph{compatibility} on the class of capped endo-permutation $RP$-lattices gives rise to a finitely generated abelian group $D_{R}(P)$, called the \emph{Dade group} of~$P$, whose multiplication is induced by the tensor product $\otimes_R$. For details, we refer the reader to~\cite{Dade78a} or~\cite[\S27-29]{ThevenazBook}.\par

If $P\leq G$ is a $p$-subgroup, we write $D_R(P)^{G-st}$ for the set of \emph{$G$-stable} elements of $D_R(P)$, i.e. the set of equivalence classes $[L]\in D_R(P)$ such that 
$$\Res^P_{\lconj{x}{P}\cap P}([L])=\Res^{\lconj{x}{P}}_{\lconj{x}{P}\cap P}\circ \, c_x([L])\in D_R(\lconj{x}{P}\cap P) \,,\quad \forall\,x\in G\,,$$
where $c_x$ denotes conjugation by~$x$.

\bigskip 

The following results can be found in Urfer \cite{URFER07} for the case $R=k$, under the additional assumption that $k$ is algebraically closed.
However, it is straightforward to prove that they hold for an arbitrary field~$k$ of characteristic~$p$, and also in case $R=\cO$. 

\begin{rem}\label{rem:stability}
It follows easily from the definitions that the class of endo-$p$-permutation $RG$-lattices is closed under taking direct summands, $R$-duals, tensor products over~$R$, (relative) Heller translates, restriction to a subgroup, and tensor induction to an overgroup. However, this class is not closed under induction, nor under direct sums.
\end{rem}

Two endo-$p$-permutation $RG$-lattices are called \emph{compatible} if their direct sum is an endo-$p$-permutation $RG$-lattice.

\begin{lem}[{}{\cite[Lemma 1.3]{URFER07}}]\label{lem:induction}
Let $H\leq G$ and $L$ be an endo-$p$-permutation $RH$-lattice. Then $\Ind_H^{G}(L)$ is an endo-$p$-permutation $RG$-lattice if and only if $\Res^H_{\lconj{x}{H}\cap H}(L)$ and $\Res^{\lconj{x}{H}}_{\lconj{x}{H}\cap H}(\lconj{x}{L})$ are compatible for each $x\in G$.
\end{lem}

\begin{thm}[{}{\cite[Theorem 1.5]{URFER07}}] \label{thm:Urfer}
An indecomposable $RG$-lattice $L$ with vertex $P$ and $RP$-source $S$ is an endo-$p$-permutation $RG$-lattice if and only if $S$ is a capped  endo-permutation $RP$-lattice such that $[S]\in D_R(P)^{G-st}$.
Moreover, in this case $\Ind_P^G(S)$ is an endo-$p$-permutation $RG$-lattice.
\end{thm}

\section{Preserving indecomposability and vertices by reduction modulo~$\fp$}

For an $\cO G$-lattice $L$, the reduction modulo $\fp$ of $L$ is 
$$L/\mathfrak{p}L\cong k\otimes_{\cO}L\,.$$
Note that $k\otimes_{\cO}\End_{\cO}(L)\cong \End_k(L/\mathfrak{p}L)$.
A $kG$-module $M$ is said to be \emph{liftable} if there exists an $\cO G$-lattice $\wh{M}$ such that $M\cong \wh{M}/\mathfrak{p}\wh{M}$.

\begin{lem}\label{lem:iso}
Let $L$ be an endo-$p$-permutation $\cO G$-lattice and write $A:=\End_{\mathcal{O}}(L)$. Then the natural homomorphism $k\otimes_{\cO}A^G\lra (k\otimes_{\cO}A)^G$ is an isomorphism of $k$-algebras.
\end{lem}

\begin{proof}
To begin with, consider a transitive permutation $\cO G$-lattice $U=\Ind_{Q}^{G}(\cO)$.
Then $Q\leq G$ is the stabiliser of $x=1_G\otimes 1_{\cO}$, so that
$$\{gx\mid g\in[G/Q]\}$$
is a $G$-invariant $\cO$-basis of $U$ and $U^G\cong \cO(\sum_{g\in[G/Q]}gx)$. It follows that 
$$\{1_k\otimes gx\mid g\in[G/Q]\}$$
is a $G$-invariant $k$-basis of $k\otimes_{\cO} U$ and 
$(k\otimes_{\cO}U)^G=k(\sum_{g\in[G/Q]}1\otimes gx)$.
Therefore the restriction of the canonical surjection $U\lra k\otimes_{\cO} U$ to the submodule $U^G$ of $G$-fixed points of $U$ has image $(k\otimes_{\cO} U)^G$ with kernel equal to $\fp U^G$. Hence the canonical homomorphism 
$$k\otimes_{\cO}U^G\lra(k\otimes_{\cO}U)^G$$
 is an isomorphism.
Because taking fixed points commutes with direct sums, the latter isomorphism holds as well for every $p$-permutation $\cO G$-lattice $U$.
Therefore, writing $A=\bigoplus_{i=1}^{m}U_i$ as a direct sum of indecomposable $p$-permutation $\cO G$-lattices, we obtain that the canonical homomorphism
$$k\otimes_{\cO}A^G\cong  \bigoplus_{i=1}^{m} \,k\otimes_{\cO} U_i^G\quad \lra \quad  \bigoplus_{i=1}^{m}\,(k\otimes_{\cO}U_i)^G\cong (k\otimes_{\cO}A)^G$$
is an isomorphism.
\end{proof}

The following characterization of vertices is well-known, but we include a proof for completeness.

\begin{lem}\label{lem:vertex}
Let $R\in\{\cO,k\}$ and let $L$ be an indecomposable $RG$-lattice.
Let $L^\vee=\Hom_R(L,R)$ denote the $R$-dual of~$L$ and let
$$\End_R(L)\cong L\otimes_RL^{\vee} \cong U_1\oplus\cdots\oplus U_n$$
be a decomposition of~$L\otimes_RL^{\vee}$ into indecomposable summands.
Then a $p$-subgroup $P$ of~$G$ is a vertex of~$L$ if and only if every $U_i$ has a vertex contained in~$P$ and one of them has vertex~$P$.
\end{lem}

\begin{proof}
Suppose $L$ has vertex~$P$. Then $L$ is projective relative to~$P$ and, by tensoring with~$L^\vee$, we see that $L\otimes_RL^{\vee}$ is projective relative to~$P$, and therefore so are $U_1,\ldots, U_n$. In other words, $P$ contains a vertex of $U_i$ for each $1\leq i\leq n$. 
Now $L$ is isomorphic to a direct summand of $L\otimes_RL^{\vee}\otimes_RL$ because the evaluation map
$$L\otimes_RL^{\vee}\otimes_RL \longrightarrow L \,, \qquad x\otimes \psi \otimes y \mapsto \psi(x)y$$
splits via $y\mapsto \sum_{i=1}^n y\otimes v_i^\vee\otimes v_i$, where $\{v_1, \ldots,v_n\}$ is an $R$-basis of~$L$ and $\{v_1^\vee, \ldots,v_n^\vee\}$ is the dual basis. Therefore $L$ is isomorphic to a direct summand of some $U_i\otimes_RL$ (by the Krull-Schmidt theorem).
If, for each $1\leq i \leq n$, a vertex of~$U_i$ was strictly contained in~$P$, then $U_i\otimes_RL$ would be projective relative to a proper subgroup of~$P$, hence the direct summand $L$ would also be projective relative to a proper subgroup of~$P$, a contradiction. This proves that, for some~$i$, a vertex of~$U_i$ is equal to~$P$.\par

Suppose conversely that every $U_i$ has a vertex contained in~$P$ and one of them has vertex~$P$. Let $Q$ be a vertex of~$L$. By the first part of the proof, every $U_i$ has a vertex contained in~$Q$ and one of them has vertex~$Q$. This forces $Q$ to be equal to~$P$ up to conjugation.
\end{proof}

\begin{prop}\label{prop:indec}
If $L$ is an indecomposable endo-$p$-permutation $\cO G$-lattice with vertex~$P\leq G$, then  $L/\fp L$ is an indecomposable endo-$p$-permutation $kG$-module with vertex~$P$.
\end{prop}

\begin{proof}
Set $A:=\End_{\cO}(L)$, so that  $A^G=\End_{\cO G}(L)$. First we prove that $\End_{kG}(L/\fp L)=(k\otimes_{\cO} A)^G$ is a local algebra. Write $\psi:A^G\lra A^G/\mathfrak{p}A^G$ for the canonical homomorphism.
By Nakayama's Lemma $\fp A^G\subseteq J(A^G)$, so that any maximal left ideal of $A^G$ contains $\fp A^G$. Therefore  
\begin{equation*}
\begin{split}
\psi^{-1}(J(A^G/\mathfrak{p}A^G))&=\psi^{-1}\left(\bigcap_{\fm\in \text{Maxl}(A^G/\mathfrak{p}A^G)} \fm \right) = \bigcap_{\substack{\fa\in \text{Maxl}(A^G)\\ \fa\supseteq \fp A^G}} \fa=J(A^G)\,,
\end{split}
\end{equation*}
where $\text{Maxl}$ denotes the set of  maximal left ideals of the considered ring.
Thus $\psi$ induces an isomorphism  $A^G/J(A^G)\cong (k\otimes_{\cO} A^G) / J(k\otimes_{\cO}A^G)$.
Now $k\otimes_{\cO}A^G\cong (k\otimes_{\cO}A)^G$ as $k$-algebras, by Lemma~\ref{lem:iso}.
Therefore it follows that
$$\End_{kG}(L/\fp L)/J(\End_{kG}(L/\fp L)) \cong
(k\otimes_{\cO} A)^G / J((k\otimes_{\cO}A)^G)\cong A^G/J(A^G) \,.$$
This is a skew-field since we assume that $L$ is indecomposable. Hence $L/\fp L$ is indecomposable.
\par

For the second claim, let $P$ be a vertex of~$L$. 
Let $L^\vee$ denote the $\cO$-dual of~$L$ and consider  a decomposition of $\End_{\cO}(L)$ into indecomposable summands
$$\End_{\cO}(L)\cong L\otimes_{\cO}L^{\vee} \cong U_1\oplus\cdots\oplus U_n\,.$$
Then there is also a decomposition
$$\End_k(L/\fp L)\cong k\otimes_{\cO} \End_{\cO}(L)\cong U_1/\fp U_1\oplus\cdots\oplus U_n/\fp U_n\,.$$
Since $L$ is an endo-$p$-permutation $\cO G$-lattice, $U_i$ is a $p$-permutation module
for each $1\leq i \leq n$. Therefore the module $U_i/\fp U_i$ is indecomposable and the vertices of $U_i$ and $U_i/\fp U_i$ are the same (see~\cite[Proposition~27.11]{ThevenazBook}).
By Lemma~\ref{lem:vertex}, every $U_i$ has a vertex contained in~$P$ and one of them has vertex~$P$.
Therefore every $U_i/\fp U_i$ has a vertex contained in~$P$ and one of them has vertex~$P$.
By Lemma~\ref{lem:vertex} again, $P$ is a vertex of~$L/\fp L$.
\end{proof}

\section{Lifting endo-$p$-permutation $kG$-modules}

We are going to use the fact that the sources of endo-$p$-permutation $kG$-modules are liftable. However, a random lift of the sources will not suffice and our next lemma deals with this question.

\begin{lem}\label{lem:G-st} 
Let $P$ be a  $p$-group. If $S$ is an indecomposable endo-permutation $kP$-module with vertex~$P$ such that $[S]\in D_k(P)^{G-st}$, then there exists an endo-permutation $\cO P$-lattice $\wh{S}$ lifting $S$ such that  $[\wh{S}]\in D_{\cO}(P)^{G-st}$.
\end{lem}

\begin{proof}
As a consequence of the classification of endo-permutation modules, Bouc proved that every endo-permutation $kP$-module is liftable \cite[Corollary 8.5]{BOUC06}. Therefore $S$ is liftable to an $\cO P$-lattice $\wh{S}$, i.e. $\wh{S}/\fp\wh{S}\cong S$. Note that $\wh{S}$ is not unique because $\wh{S}\otimes_{\cO}L$ also lifts~$S$ for any one-dimensional $\cO P$-lattice~$L$. This is because $L/\fp L\cong k$ since the trivial module~$k$ is the only one-dimensional $kP$-module up to isomorphism. However, the lifted $P$-algebra $\End_{\cO}(\wh{S})$ is unique up to isomorphism and we can choose $\wh{S}$ to be the unique ${\cO}P$-lattice with determinant~$1$ which lifts~$S$ (see \cite[Lemma~28.1]{ThevenazBook}). This choice of an $\cO P$-lattice with determinant~1 is made possible because the dimension of~$\wh{S}$ is prime to~$p$ (see \cite[Corollary~28.11]{ThevenazBook}).\par

In order to prove that $[\wh{S}]$ is $G$-stable in the Dade group, we note that the determinant~1 is preserved by conjugation and by restriction. Therefore, the equality
$$\Res^P_{\lconj{x}{P}\cap P}([S])=\Res^{\lconj{x}{P}}_{\lconj{x}{P}\cap P}\circ \, c_x([S])\in D_k(\lconj{x}{P}\cap P) \,,\quad \forall\,x\in G$$
implies an equality for the unique lifts with determinant~1
$$\Res^P_{\lconj{x}{P}\cap P}([\wh{S}])=\Res^{\lconj{x}{P}}_{\lconj{x}{P}\cap P}\circ \, c_x([\wh{S}])\in D_{\cO}(\lconj{x}{P}\cap P) \,,\quad \forall\,x\in G\,.$$
This proves that $[\wh{S}]\in D_{\cO}(P)^{G-st}$, completing the proof.
\end{proof}

\begin{thm}
Let $M$ be an indecomposable endo-$p$-permutation $kG$-module, and let $P\leq G$ be a vertex of $M$. Then there exists an indecomposable  endo-$p$-permutation $\cO G$-lattice $\wh{M}$ with vertex $P$ such that $\wh{M}/\fp\wh{M}\cong M$.
\end{thm}

\begin{proof}
Let $P$ be a vertex of $M$ and $S$ be a $kP$-source of $M$.  
By Theorem~\ref{thm:Urfer}, $S$ is a capped endo-permutation $kP$-module such that $[S]\in D_k(P)^{G-st}$. By  Lemma~\ref{lem:G-st}, $S$ lifts to an endo-permutation $\cO P$-lattice $\wh{S}$ such that  $[\wh{S}]\in D_{\cO}(P)^{G-st}$. Moreover $\Ind_P^G(\wh{S})$ is an endo-$p$-permutation $\cO G$-lattice, by Lemma~\ref{lem:induction} and the fact that $[\wh{S}]$ is $G$-stable.
Now consider a decomposition of $\Ind_P^G(\wh{S})$ into indecomposable summands
$$\Ind_P^G(\wh{S})=L_1\oplus\cdots\oplus L_s \quad (s\in\IN)\,.$$
By Remark~\ref{rem:stability}, each of the lattices $L_i$ $(1\leq i\leq s)$ is an endo-$p$-permutation $\cO G$-lattice.
Then, by Proposition~\ref{prop:indec}, 
$$\Ind_P^G(S)\cong \Ind_P^G(\wh{S})/\fp\Ind_P^G(\wh{S})
\cong L_1/\fp L_1\oplus\cdots\oplus L_s/\fp L_s$$
is a decomposition of $\Ind_P^G(S)$ into indecomposable summands which preserves the vertices of the indecomposable summands.
Because $S$ is a source of~$M$, there exists an index $1\leq i\leq s$ such that $M\cong L_i/\fp L_i$.
Then $\wh M:=L_i$ lifts~$M$.
\end{proof}

\bigskip

\begin{rem}
In \cite{BK06}, Boltje and K\"{u}lshammer consider the class of \emph{modules with an endo-permutation source}, which also play a role in the study of Morita equivalences, as observed by Puig \cite{Puig}.
In recent work of Kessar and Linckelmann \cite{KessLinck17}, it is proved that in odd characteristic any Morita equivalence with an endo-permutation source is liftable from $k$ to~$\cO$, under the assumption that $k$ is algebraically closed. \par

As a typical example, we remark that simple modules for $p$-soluble groups are known to be instances of modules with an endo-permutation source (see \cite[Theorem~30.5]{ThevenazBook}) and they are also known to be liftable to characteristic zero (Fong-Swan Theorem). Urfer proved in his Ph.D. thesis \cite{URFER06} that such simple modules are endo-$p$-permutation modules in case they are not induced from proper subgroups, but in general they need not be endo-$p$-permutation. \par

One may ask whether our result extends to $kG$-modules with an endo-permutation source, i.e. whose class in the Dade group is not necessarily $G$-stable. We do not have an answer to this question.
Our proof that endo-$p$-permutation modules are liftable to characteristic zero does not seem to extend to this larger class of modules, because it relies on the fact that the endomorphism algebra is a $p$-permutation module.
\end{rem}

\bigskip

\noindent\textbf{Acknowledgments.} The authors are grateful to Nadia Mazza for useful discussions.


\end{document}